\documentclass[11pt]{amsart}
\usepackage{graphicx,enumerate}
\usepackage{amsmath,amssymb}
\usepackage{mathtools,url}
\usepackage[margin=1in]{geometry}
\usepackage{amsthm}
\usepackage[]{algorithm2e}

\usepackage[usenames,dvipsnames]{xcolor}
\definecolor{purple}{rgb}{0.3,0.0,.4}

\DeclareSymbolFont{bbold}{U}{bbold}{m}{n}
\DeclareSymbolFontAlphabet{\mathbbold}{bbold}

\theoremstyle{theorem}
\newtheorem{theorem}{Theorem}[section]

\newtheorem{lemma}[theorem]{Lemma}
 
\newtheorem{prop}[theorem]{Proposition}

\theoremstyle{definition}
\newtheorem{definition}[theorem]{Definition}

\theoremstyle{remark}

\numberwithin{theorem}{section}
\numberwithin{equation}{section}

\DeclareMathOperator{\tr}{tr}
\DeclareMathOperator{\diag}{diag}
\DeclareMathOperator{\card}{card}
\newcommand{\twonorm}[1]{\|#1\|_2}
\newcommand{\comb}[2]{{#1 \choose #2}}
\newcommand{\psd}[1]{\mathbb{S}^{#1}_+}
\newcommand{\ep}{\text{embedding property}\xspace}

\newcommand{\cone}[1]{\mathsf{K}^{#1}}
\newcommand{\cp}[1]{\mathsf{COP}^{#1}}
\newcommand{\cpp}[1]{\mathsf{CP}^{#1}}
\newcommand{\SDD}[1]{\mathsf{SDD}^{#1}}
\newcommand{\Trn}[1]{\tau_{#1}}
\newcommand{\Aug}[1]{\varepsilon_{#1}}
\newcommand{\DD}[1]{\mathsf{DD}_{#1}}
\newcommand{\SOC}{\mathsf{SOC}}

\newcommand{\nn}[1]{\mathcal{N}^{#1}_+}
\newcommand{\Ho}[1]{\mathcal{HO}^{#1}}

\DeclareMathOperator{\interior}{int}
\newcommand{\SOS}{\mathsf{SOS}}
 
\newcommand{\ddsos}{\mathsf{DDSOS}}
\newcommand{\optx}{X^{\star}}
\newcommand{\opty}{y^\star}
\newcommand{\optz}{Z^{\star}}
\newcommand{\tp }{{\scriptscriptstyle\mathsf{T}}}

\newcommand{\marka}[1]{\overset{(a)}{#1}}

\begin{document}
	\title{Higher-Order Cone Programming}
	\author{Lijun Ding}
	\address{Department of Statistics, University of Chicago, Chicago, IL 60637-1514.}
	\email{ld446@cornell.edu}
	\author{Lek-Heng Lim}
	\address{Computational and Applied Mathematics Initiative, Department of Statistics,
		University of Chicago, Chicago, IL 60637-1514.}
	\email{lekheng@galton.uchicago.edu}
\begin{abstract}
We introduce a conic embedding condition that gives a hierarchy of cones and cone programs. This condition is satisfied by a large number of convex cones including the cone of copositive matrices, the cone of completely positive matrices, and all symmetric cones. We discuss properties of the intermediate cones and conic programs in the hierarchy. In particular, we demonstrate how this embedding condition gives rise to a family of cone programs that interpolates between LP, SOCP, and SDP. This family of $k$th order cones may be realized either as cones of $n$-by-$n$ symmetric matrices or as cones of $n$-variate even degree polynomials. The cases $k = 1, 2, n$ then correspond to LP, SOCP, SDP; or, in the language of polynomial optimization, to DSOS, SDSOS, SOS.
\end{abstract}
	\maketitle

\section{Introduction}

Given a convex proper cone we will show how to construct a hierarchy of cones with associated cone programs, provided that a certain embedding property (defined below) is satisfied. This generalizes the work of Ahmadi and Majumdar in \cite{ahmadi2017dsos} where they constructed a sequence of polynomial conic programs, particularly the DSOS and SDSOS conic programs, to approximate the SOS cone program. We will show how such a construction can be carried out for a large number of conic programming problems including:
\begin{enumerate}[\upshape (i)]
\item the nonnegative orthant;
\item\label{it:soc} the second-order cone;
\item\label{it:psd} the cone of symmetric positive semidefinite matrices;
\item the cone of copositive matrices;
\item the cone of completely positive matrices;
\item all symmetric cones, i.e., any cone is constructed out of a direct sum of \eqref{it:soc}, \eqref{it:psd}, or the cones of Hermitian positive semidefinite matrices over $\mathbb{C}$, $\mathbb{H}$, and $\mathbb{O}$ (quaternions and octonions);
\item any norm cones where the norm satisfies a consistency condition, which includes $l^p$-norms, Schatten and Ky Fan norms, operator $(p,q)$-norms, etc.
\end{enumerate}
For each of these cones, we can build a sequence of intermediate cones and conic programs in the hierarchy. In the case of \eqref{it:soc}, we obtain a family of cone programs that interpolates between LP, SOCP, and SDP. This family of $k$th order cones may be realized either as cones of $n$-by-$n$ symmetric matrices or as cones of $n$-variate even degree polynomials. The cases $k = 1, 2, n$ then correspond to LP, SOCP, SDP; or, in the language of polynomial optimization, to DSOS, SDSOS, SOS.

\subsection*{Notations} Throughout this article, we write $\mathbb{N} \coloneqq \{1,2,3,\ldots\}$ for the set of positive integers. The skew field of quaternions will be denoted as $\mathbb{H}$ and the division ring of octonions as $\mathbb{O}$.  We will slightly abuse terminologies and refer to $\mathbb{R}$, $\mathbb{C}$, $\mathbb{H}$, $\mathbb{O}$ as `fields.' We will write $\mathbb{S}^{d}_\mathbb{F}$ for the $\mathbb{F}$-vector space (or, strictly speaking, $\mathbb{F}$-module when $\mathbb{F}$ is not a field) of $d \times d$ Hermitian matrices over $\mathbb{F} = \mathbb{R},\mathbb{C}, \mathbb{H}, \mathbb{O}$. When the choice of $\mathbb{F}$ is implicit or immaterial, we will just write  $\mathbb{S}^{d}$. For a vector $x\in \mathbb{F}^d$, the notation $x\geq 0$ means each component of $x$ is greater or equal to $0$.

We write $[d] \coloneqq \{1,\dots,d\}$ for any $d\in \mathbb{N}$. We denote the 
set of all increasing sequences of length $k$ in $[d]$ as $\comb{[d]}{k}=\{(i_1,\dots,i_k)\mid  1\leq \,i_1<\dots <i_k\leq d\}$. 

  For a matrix $A = [a_{ij}]_{ij}\in\mathbb{S}^d$, we write $\tr(A)= \sum_{i=1}^d a_{ii}$. The inner product $\langle \cdot ,\cdot \rangle :\mathbb{S}^d\times \mathbb{S}^d \rightarrow \mathbb{R}$ we use in this article is the standard trace inner product $\langle A, B\rangle = \tr(AB)$. The topology is then defined via the distance metric induced by the trace inner product. We write the interior of a set $S \subset \mathbb{F}^{d}$ as $\interior(S)$. 
\section{Conic embedding property}

To standardize our terminologies, the cones in this article will all be represented as cones of symmetric matrices over some field $\mathbb{F}$; although we will see that this is hardly a limitation ---  conic programs involving cones in other common $\mathbb{F}$-vector spaces, e.g., of vectors in $\mathbb{F}^n$ or polynomials in $\mathbb{F}[x]$ or $\mathbb{F}$-valued functions on some set, can often be transformed to a symmetric matrix setting.

We start by defining two linear maps. Let $k \le d$ be positive integers. For $\{i_1,\dots,i_k\} \in \comb{[d]}{k}$, i.e., $1\leq i_1<\dots<i_k\leq d$, the \emph{truncation operator} is the projection $\Trn{i_1 \cdots i_k}^d: \mathbb{S}^{d} \rightarrow \mathbb{S}^{k}$ defined by
\[
 \Trn{i_1 \cdots i_k}^d(Z) \coloneqq\begin{bmatrix}
	z_{i_1i_1} & \dots & z_{i_1i_k}\\
	z_{i_2i_1} & \dots & z_{i_2i_k}\\
	& \ddots & \\
	z_{i_k i_1} & \dots &z_{i_ki_k}
	\end{bmatrix}
\]
for any $Z\in \mathbb{S}^{d}$; the \emph{lift operator}  is the injection $\Aug{i_1 \cdots i_k}^d : \mathbb{S}^{k}\rightarrow \mathbb{S}^{d}$ defined by
\[
[\Aug{i_1 \cdots i_k}^d(X)]_{i_p i_q} =
\begin{cases}
 x_{pq} & p,q\in \{1,\dots,k\} ,\\
0 & \text{otherwise}.
\end{cases}
\]
for any $X\in \mathbb{S}^{k}$. In other words, the truncation operator takes a $d \times d$ matrix to its $k \times k$ submatrix; whereas the lifting operator takes a $k \times k$ matrix and embed it as a $d \times d$ matrix by filling-in the extra entries as zeros. Clearly for a fixed index set $\{i_1,\dots,i_k\}$,  $\Trn{i_1 \cdots i_k}^d$ is a left inverse of $\Aug{i_1 \cdots i_k}^d$, i.e.,
\[
\Trn{i_1 \cdots i_k}^d \circ \Aug{i_1 \cdots i_k}^d = \operatorname{id}_{\mathbb{S}^{k}}.
\]
We now state our \ep.   
	\begin{definition}
Let $\mathbb{F} = \mathbb{R}$, $\mathbb{C}$, $\mathbb{H}$, or $\mathbb{O}$ and $\mathbb{S}^k =\mathbb{S}^k_\mathbb{F}$. Let $k_0 \in  \mathbb{N}$ and $\{\cone{k} : k\in \mathbb{N},\; k \ge k_0\}$ be a sequence of convex proper cones where $\cone{k} \subseteq \mathbb{S}^k$ for each $k \ge k_0$. We say that the sequence $\{\cone{k} \}_{k=k_0}^\infty$ satisfies the  \emph{\ep} with \emph{index map}
\begin{equation}\label{eq:index}
I:\{(d,k)\in \mathbb{N}\times \mathbb{N}\mid d\geq k\} \rightarrow \bigcup_{k_0\le k \le d}\comb{[d]}{k}
\end{equation}
if for any $d\geq k\geq k_0$,  $(i_1,\dots,i_k)\in I(d,k)$, we have
\[
	\Trn{i_1i_2 \cdots i_k}^d (Z) \in \cone{k} \qquad \text{and} \qquad
	\Aug{i_1i_2 \cdots i_k}^d (X) \in \cone{d}
\]
for all $Z\in \cone{d}$ and  $ X\in \cone{k}$.
	\end{definition}
We caution our reader that the ``higher-order cones'' in the title of this article do not refer to   $\{\cone{k}\}_{k=k_0}^\infty$ but will be constructed out of these cones.
In several instances, the index map is given simply by
\[
I(d,k) = \comb{[d]}{k},
\]
and in which case we will drop any reference to the index map and just say that $\{\cone{k}\}_{k=k_0}^\infty$ satisfies the \ep. If in addition $k_0=1$, we will say that $\{\cone{k}\}_{k=1}^\infty$ satisfies the \ep \emph{thoroughly}.  

The embedding property simply says for a $d \times d$ matrix $Z\in \cone{d}$, its $k \times k$ principle  submatrix belongs to the lower dimension cone $\cone{k}$; conversely, for a  $k \times k$ matrix $X\in \cone{k}$, embedding as it a principle submatrix of a $d \times d$  matrix with all other entries  set to be zero gives a matrix in $\cone{d}$.

A simple example is the cone of symmetric diagonally dominant matrices with nonnegative diagonals,
\[
\DD{d} \coloneqq  \Bigl\{  M \in \mathbb{S}^d : m_{ii}\geq \sum\nolimits_{j\ne i} |m_{ij}|, \; i=1,\dots,d \Bigr\},
\]
where it is easy to see that $\{ \DD{k}\}_{k=1}^\infty$ satisfies the \ep thoroughly. We will see many more examples of cones satisfying the \ep over the next few sections. 

We may now define the higher order cones in the title of this article. They are obtained by lifting cones in lower dimension to higher dimension. The benefit is that though the cones defined are in high dimension, they are expressible by cones in lower dimension and property of cones in lower dimension might be utilized. These higher cone might be served as an inner approximation of cones in high dimension. 

As usual, in the following we let $\mathbb{F} = \mathbb{R}$, $\mathbb{C}$, $\mathbb{H}$, or $\mathbb{O}$ and write $\mathbb{S}^d = \mathbb{S}^d_\mathbb{F}$.
\begin{definition}
Let $\{\cone{k}\}_{k=k_0}^\infty$ be a sequence of cones that satisfies the \ep with index map $I$. The \emph{$k$th order cone} with index set $J\subseteq I(d,k)$ induced by $\cone{k}$ is
\[
 \cone{d}_{k}(J) \coloneqq \Bigl\{ M \in  \mathbb{S}^d : M=\sum\nolimits_{(i_1,\dots,i_k)\in J} \Aug{i_1 \cdots i_k}^d (M_{i_1 \cdots i_k}),\; M_{i_1 \cdots i_k}\in \cone{k} \Bigr\}.
\]
If $J = I(d,k)$, we will just write  $\cone{d}_k$ for $\cone{d}_{k}(J)$.
\end{definition}


We will establish some basic properties of  higher order cones. 
\begin{prop}\label{proposition: koc} Let $\{\cone{d}\}_{k=k_0}^\infty$ satisfy the \ep with index mapping $I$. Then the following properties hold:
	\begin{enumerate}[\upshape (i)]
		\item \emph{Nested cones:} Suppose a sequence of index sets $\{J_k\}_{k=k_0}^d$, $J_k \subset \comb{[d]}{k}$ satisfies that for any  $k\geq k_0$ and any $s\in J_k$, there is an $s' \in J_{k+1}$ such that all the components of $s$ appears in $s'$ (This property is satisfied by $\comb{[d]}{k}$). Then we have 
		\[ \cone{d}_{k_0}(J_{k_0}) \subseteq \cone{d}_{k_0+1}(J_{k_0+1})\subseteq \dots \subseteq \cone{d}_{d}(J_{d}).\]
		In particular, if $\{I(d,k)\}_{k=k_0}^d$ is such sequence of index sets, then for every $d\geq k_0$, we have 
		\[ \cone{d}_{k_0}\subseteq \cone{d}_{k_0+1}\subseteq \dots \subseteq \cone{d}_{d}.\]
		\item \emph{Dual cones:} the dual cone of $\cone{d}_k(J)$ is 
		\[(\cone{d}_k(J))^* = \{A \in \mathbb{S}^d  : A\,\text{is Hermitian and}\,\text{for all}\, (i_1,\dots,i_k)\in J,\,\Trn{i_1 \cdots i_k} (A) \in (\cone{k})^*\} .\]
		\item \emph{Membership:} If $I(d,k) =\comb{d}{k}$ for every $d$, we have $X_1\in \cone{t}_k, X_2\in \cone{s}_k \iff \diag(X_1,X_2)\in \cone{t+s}_k$, ditto for the dual cones of $\cone{d}_k$.
		\item \emph{Inheritance}: It the \ep is satisfied throughly by $\{\cone{k}\}_{k=1}^\infty$, then for each $k\geq 1$, the sequence of cones $\{\cone{l}_{k}\}_{l=k}^\infty$ satisfies the \ep.
	\end{enumerate} 
\end{prop}
\begin{proof}
\begin{enumerate}[\upshape (i)]
	\item  Consider $k_0+i$ and $k_0+i+1$ where $0\leq i\leq d-k_0-1$. The cones $\cone{d}_{k_0+i}(J_{k_0+i})$ and $\cone{d}_{k_0+i+1}(J_{k_0+i+1})$ can be expressed as 
	\begin{equation}\label{eq: Kk0decomposition}
	\begin{aligned} 
	 \cone{d}_{k_0+i}(J_{k_0+i}) = \sum_{(j_1,\dots,j_{k_0+i})\in J_{k_0+i}}\Aug{j_1\dots j_{k_0+i}}^d (\cone{k_0+i})
	 \end{aligned}
	 \end{equation}
	and 
		\begin{equation}\label{eq: Kk0+1decomposition}
		\begin{aligned} 
		 \cone{d}_{k_0+i+1}(J_{k_0+i+1}) = \sum_{(j_1,\dots,j_{k_0+i},j_{k_0+i+1})\in J_{k_0+i+1} }\Aug{j_1\dots j_{k_0+i} j_{k_0+i+1}}^d (\cone{k_0+i+1})
		\end{aligned}
	\end{equation} 
	By the assumption on $\{J_k\}_{k=k_0}^d$,  we know for each $(j_1,\dots,j_{k_0+i})\in J_{k_0+i}$, there is some $j$ such that $\{j_1,\dots,j_{k_0+i},j\}$ after ordering is in $ J_{k_0+i+1}$. Without loss of generality, we may assume $j$ is the largest among $\{j_1,\dots,j_{k_0+i},j\}$. Thus 
	\begin{align*}\Aug{j_1\dots j_{k_0+i}}^d (\cone{k_0+i}) & =\Aug{j_1\dots j_{k_0+i}j}^d (\begin{bmatrix}
	\cone{k_0+i} &  \\  & 0 
	\end{bmatrix})  \\ 
 & \marka{\subset}	\Aug{j_1\dots j_{k_0+i}j}^d(\cone{k_0+i+1}),
	\end{align*}
	where (a) is because of $\begin{bmatrix}
	\cone{k_0+i} &  \\  & 0 
	\end{bmatrix}\subseteq \cone{k_0+i+1}$ using the \ep. Thus we see each summand in the decomposition \eqref{eq: Kk0decomposition} is a subset of a summand in the decomposition of \eqref{eq: Kk0+1decomposition}. Using the conic property that $a,b \in \cone{k_0+i+1} \implies a+b \in \cone{k_0+i+1}$, we see indeed 
	\[ \cone{d}_{k_0+i} \subseteq \cone{d}_{k_0+i}.\]
	Since $i$ is arbitrary, we see we have the cones are nested.
	\item  We use the following simple fact \cite[Corollary 16.3.2]{rockafellar1970convex} that for convex cone $K_1,K_2$, \[(K_1+K_2)^* = K_1^*\cap K_2^*.\]
	By definition, $\cone{d}_{k}(J)$ can be expressed as 
	\[ \cone{d}_k(J) = \sum_{(i_1,\dots, i_k)\in J}\Aug{i_1 \cdots i_k}^d(\cone{k}),\]
	where each $\Aug{i_1 \cdots i_k}^d(\cone{k})$ is a convex cone  in $\mathbb{S}^d$. The dual cone of  $\Aug{i_1 \cdots i_k}^d(\cone{k})$	is \[( \Aug{i_1 \cdots i_k}^d(\cone{k}))^* = \{A\in \mathbb{S}^d: \Trn{i_1 \cdots i_k}(A)\in (\cone{k})^*\}.\]
	Applying previous fact, we get the characterization of the dual cone. 
	\item We first show that  $X_1\in \cone{t}_k, X_2\in \cone{s}_k \implies \diag(X_1,X_2)\in \cone{s+t}_k$. We know there are $M_{i_1 \cdots i_k}$,$Y_{j_1\dots j_k}\in \cone{k}$ such that 
	\[ X_1 = \sum_{(i_1,\dots,i_k)\in\comb{[t]}{k}}\Aug{i_1 \cdots i_k}^s(M_{i_1 \cdots i_k}),\]
and 
	\[ X_2 = \sum_{(j_1,\dots,j_k)\in \comb{[s]}{k}}\Aug{j_1\dots j_k}^t(Y_{j_1\dots j_k}).\] Thus 
	\begin{align*}
	\diag(X_1,X_2) =&\sum_{(i_1,\dots,i_k)\in\comb{[t]}{k}}\diag(\Aug{i_1 \cdots i_k}^s(M_{i_1 \cdots i_k}),0) \\&+ \sum_{(j_1,\dots,j_k)\in \comb{[s]}{k}}\diag(0,\Aug{j_1\dots j_k}^t(Y_{j_1\dots j_k}))\\
	=&   \sum_{(i_1,\dots,i_k)\in\comb{[t]}{k}}\Aug{i_1 \cdots i_k}^{s+t}(M_{i_1 \cdots i_k})\\
	& + \sum_{(j_1,\dots,j_k)\in \comb{[s]}{k}}\Aug{j_1\dots j_k}^{s+t}(Y_{j_1\dots j_k}).
	\end{align*}
	Since $\comb{[s]}{k},\comb{[t]}{k} \subseteq \comb{[s+t]}{k}$, we see the above indeed gives a valid decomposition of $k$th order cone induced by $\{\cone{k}\}_{k=1}^\infty$. 
	Now suppose $\diag(X_1,X_2)\in \cone{s+t}_k$. This gives 
	\begin{align*}
	\diag(X_1,X_2) = \sum_{(i_1,\dots, i_k)\in \comb{[s+t]}{k} }\Aug{i_1 \cdots i_k}^{s+t}(Z_{i_1 \cdots i_k})
	\end{align*}
	where $Z_{i_1 \cdots i_k} \in \cone{k}$. Apply $\Trn{1,2,\dots,s}$ and $\Trn{s+1,s+2,\dots,s+t}$ to both sides of the above equality gives  valid decompositions of $X_1 \in \cone{s}_k$ and $X_2 \in \cone{t}_k$ due to the \ep.
	
	For the dual cones, note that if $X\in (\cone{k})^*$, then for each $l\leq k$, 
	$\Trn{i_1\dots i_{l}}(X) \in (\cone{l})^* $ because of the \ep. The rest of the proof is similar to previous one. 
	\item  Fix $k\leq l<m$. Consider an increasing sequence $(i_1,\dots i_l)\in \comb{[m]}{l}$ and any $X \in \cone{l}_k, Z\in \cone{m}_k$. Then there are some $M_{j_1\dots j_k}\in \cone{k},Y_{n_1\dots n_k}\in \cone{k}$ such that 
	\begin{align*}
	\Aug{i_1,\dots,i_l} ^m(X) & = \Aug{i_1,\dots, i_l}^m\biggr(\sum_{(j_1,\dots,j_k)\in \comb{[l]}{k}}\Aug{j_1\dots j_k }^l(M_{j_1\dots j_k})\biggr) \\
	& =\sum_{(j_1,\dots,j_k)\in \comb{[l]}{k}}\Aug{i_1,\dots, i_l}^m\biggr(\Aug{j_1\dots j_k }^l(M_{j_1\dots j_k})\biggr),
	\end{align*}
	and 
	\begin{align*}
	\Trn{i_1\dots i_l} ^m(Z) &= \Trn{i_1\dots i_l}^m\biggr(\sum_{(n_1,\dots,n_k)\in \comb{[m]}{k}}\Aug{n_1\dots n_k }^m(Y_{n_1\dots n_k})\biggr) \\
	& =  \sum_{(n_1,\dots,n_k)\in\comb{[m]}{k}}\Trn{i_1\dots i_l}^m\biggr(\Aug{n_1\dots n_k }^m(Y_{n_1\dots n_k})\biggr)
	\end{align*}
	Since $\Aug{i_1,\dots, i_l}^m\circ (\Aug{j_1\dots j_k }^l) = \Aug{i_{j_1}\dots i_{j_k}}^m$, we see $\Aug{i_1,\dots,i_l} ^m(X) $ is indeed a member of $\cone{m}_{k}$.  Using the \ep for each $\Trn{i_1\dots i_l}^m\biggr(\Aug{n_1\dots n_k }^m(Y_{n_1\dots n_k})\biggr)$,  we see $\Trn{i_1\dots i_l} ^m(Z) \in \cone{l}_{k}$.  
\end{enumerate}
\end{proof}

Given the definition of higher order cone and dual cone, we can consider their corresponding conic programs. We assume the underlying filed is real for simplicity. More precisely, the $k$th order cone program (standard form) is
\begin{equation}\label{program:kocpstandard}
\begin{aligned} 
& \text{minimize}& & \tr(A_0X) \\
& \text{subject to}\;& & \tr(A_iX)=b_i ,\quad i =1,\dots,p \\
& & & X\in \cone{d}_k(J)
\end{aligned}
\end{equation}
where $A_1,\dots,A_p \in\mathbb{R}^{n\times n}$, $ b_1,\dots,b_p\in \mathbb{R}$

Alternatively, $k$th order cone program (inequality form)  is
\begin{equation}\label{program:kocpinequality}
\begin{aligned}
&\text{minimize} & & q^\tp  x \\
& \text{subject to}& & x_1P_1+x_2P_2+\dots+x_kP_k+P_0\in \cone{d}_k(J)
\end{aligned}
\end{equation}
where $P_0,\dots,P_k \in \mathbb{S}^d$ and $q\in \mathbb{R}^n$. The constraint here is called linear matrix inequality (LMI).

We call these programs \emph{$k$OCP} induced by $\cone{k}$ with set $J$ and simply \emph{$k$OCP} if the underlying cone $\cone{k}$ is clear from the context and $J = I(k,d)$. We note that the ambient dimension $d$ might change from problem to problem as the case of semidefinite programming where the ambient dimension $d$ is not specified, i.e., we write $X\succeq 0$ meaning $X$ is positive semidefinite but did not specify the size of $X$. 

If the nested cones property is satisfied by the underlying cone $\cone{k}$ (which is true when $I(d,k)$ satisfies the condition of first item, Nested Cones, of Proposition \ref{proposition: koc}), the above program serves as inner approximation of $\cone{d}$ program.  

We state an equivalence theorem of the two form when $\cone{k}$ satisfies the \ep thoroughly. 

\begin{theorem}\label{thm:equivalencesdfineq}
	If $\{\cone{k}\}_{k=1}^{\infty}$ satisfies the \ep thoroughly, then the inequality form and the standard form are equivalent. 
\end{theorem}
\begin{proof}
	Without loss of generality, we assume that $\cone{1} = \mathbb{R}_+$ (because one dimensional proper cone is either $\mathbb{R}_+$ or $\mathbb{R}_-$) and $p\geq k$ (we can repeat a few constraints if $p<k$).
	
	 By Lemma \ref{lemma: geqzerokoc} proved in the Appendix, we find that for any $x\in \mathbb{F}^d$, \[\diag(x)\in \cone{k}_k\iff x\geq 0,\] where $x\geq 0$ means each component of $x$ is greater or equal to $0$.
	
	We first prove the direction from the standard form to inequality form, i.e., \eqref{program:kocpstandard} to \eqref{program:kocpinequality}:
	
	By treating $X$ as a long vector, the objective and the conic constraint $X\in \cone{d}_k$ can be transformed in a standard way. Indeed, the objective is just $\sum_{ij}(A_0)_{ij}x_{ij}$.  For conic constraint, we have 
	\begin{equation}\label{eq:conicconstrainttoLMI}
	\begin{aligned} 
	X\in \cone{d}_k \iff \sum_{j>k}x_{jk}(E_{jk}+E_{kj})+\sum_{j=1}^n x_{jj}E_{jj}\in \cone{d}_k,
	\end{aligned}
	\end{equation}
	 where $E_{jk}$ are the matrices with only non-zero entry $1$ at $(j,k)$th entry. The linear constraint 
	$\tr(A_iX)=b_i$ can be encoded by 
	\begin{equation}\label{eq:lineartoLMI}
	\begin{aligned}
	\diag([\tr(A_iX)-b_i]_{i=1}^p) \in \cone{k}_k,\quad \diag([b_i-\tr(A_iX)]_{i=1}^p) \in \cone{k}_k.
	\end{aligned}
	\end{equation}
	Finally, using membership property in Lemma \ref{proposition: koc}, the transformed linear constraints \eqref{eq:lineartoLMI} and the transformed conic constraint \eqref{eq:conicconstrainttoLMI} can be made into one big $k$th order cone linear matrix inequality.

	We now prove the direction from inequality form to standard form, i.e., \eqref{program:kocpinequality} to \eqref{program:kocpstandard}:
	
	First we can write $x = x^+-x^-$ as two non-negative vectors (element wise non-negative). Let $\bar{X} =  x_1P_1+x_2P_2+\dots+x_kP_k+P_0$, then the inequality form \eqref{program:kocpinequality} can be transformed to 
	\begin{equation*}
	\begin{aligned}
	& \underset{x^+,x^-,\bar{X}}{ \text{minimize}} & & q^\tp x^+- q^\tp x^-
	\\ &   \text{subject to}& & \sum_{i=1}^{k}(x^+_iP_i-x^-_iP_i) -\bar{X}=-P_0\\
	& & &\bar{X}  \in \cone{d}_k,\quad x^+\geq 0,\quad x^-\geq 0.
	\end{aligned}
	\end{equation*}
	It can then be transformed to \eqref{program:kocpstandard}. We may let the $X$ in \eqref{program:kocpstandard} be 
	\[
	X= \begin{bmatrix}
	\diag(x^+) & & \\
	& \diag(x^-) & \\
	& & \bar{X}\\
	\end{bmatrix}.
	\]
	
	The objective in \eqref{program:kocpstandard} then can be easiy formulated as $D_0 = \begin{bmatrix}
	\diag(q) & & \\ & -\diag(q) & \\ & & 0\\
	\end{bmatrix}$. The equality constraints are just a re-statement of  the elementwise version of  $\sum_{i=1}^{k}(x^+_iP_i-x^-_iP_i) -\bar{X}=-P_0$. So $D_i,f_i$ are setted so that $\tr(D_iX)= [\sum_{t=1}^{k}(x^+_tP_t-x^-_tP_t) -\bar{X}]_{jk}=-P_{jk}=f_i$. A total of $\frac{n(n+1)}{2}$ constraints can be obtained from this method. To enforce the $0$ in $X$, we can put more $\tr(E_{ij}X)=0$ constraints on $X$ with position index $(i,j)$ of $0$ in $X$ where $E_{ij}$ is defined as the previous part.  These linear constraints implies that  for $n\geq k$, $X\in \cone{d+2n}_k$ if and only if $x^-,x^+\geq 0, \bar{X}\in \cone{d}_k$ because the membership property
	 and Lemma \ref{lemma: geqzerokoc}. If $n<k$, we may simply repeat $x^-, x^+$ in $X$ and enforce the repetition by adding more linear constraints. 
\end{proof}
The dual $k$OCP (standard form) is 
\[
\begin{aligned} 
&  \text{minimize} & & \tr(A_0X) \\
& \text{subject to}\;& & \tr(A_iX)=b_i,\quad i=1,\dots,p \\
& & & X\in (\cone{d}_k(J))^*
\end{aligned}
\]
where $A_1,\dots,A_p \in\mathbb{R}^{n\times n}$, $ b_1,\dots,b_p\in \mathbb{R}$ and the dual $k$OCP (inequality form) is 
\[
\begin{aligned}
&\text{minimize} & & q^{\top}x \\
& \text{subject to}& & x_1P_1+x_2P_2+\dots+x_kP_k+P_0\in (\cone{d}_k(J))^*
\end{aligned}
\]
where $P_0,\dots,P_k \in \mathbb{S}^n$.

\section{Positive semidefinite cone}

Our first example is the cone of positive semidefinite matrices with dimension $d$: 
\[ \psd{d} \coloneqq \{A\in \mathbb{S}^d : A = FF^{\top} \, \text{for some}\, F\in \mathbb{R}^{d\times r},\,r\in  \mathbb{N}\}.\]

Clearly, the sequence of cones $\{\psd{k}\}_{k=1}^d$ satisfy the \ep thoroughly. The first two order cones are: 
\begin{enumerate}[\upshape (i)]
	\item $(\psd{d})_1 = \diag(\mathbb{R}_+^d) \cong \mathbb{R}_+^d$. Note that from the inheritance property, fourth item of Proposition \ref{proposition: koc}, the nonnegative orthant $\diag(\mathbb{R}_+^d) \cong\mathbb{R}_+^d$ satisfies the embedding property throughly as well, which can also be directly verified.
	\item $(\psd{d})_2  = \{A : A =\sum\nolimits_{i<j}\Aug{ij}(M^{ij}), \quad M^{ij}\in \mathbb{S}_+^2\}$.  Note this series of cone also satisfied the embedding property throughly by attaching $\mathbb{R}_+$ to the series $\{(\psd{d})_2\}_{d=2}^\infty$. 
\end{enumerate}

It turns out that the second order cone $(\psd{d})_2$ actually is the same as the set of symmetric scaled diagonally dominant matrices with nonnegative diagonals (SDD), $\SDD{d}$,
\[
\SDD{d} \coloneqq  \{  M \in \mathbb{S}^{d} : \text{there exists }  d>0, d_ia_{ii}\geq \sum\nolimits_{j \ne i} d_j|a_{ij}|, \quad \text{for all}\, i=1,\dots, d\},
\]
as shown in the following lemma, which appeared in \cite[Theorems~8 and 9]{boman2005factor} and \cite[Lemma~9]{ahmadi2017dsos}.

\begin{lemma}\label{lemma: 2oc=sdd} $(\psd{d})_2  = \SDD{d}$.
\end{lemma}

We provide a simple, different and self-contained proof of this lemma based on the following lemma which can be found in Appendix.
\begin{lemma}\label{lemma: sdd results}
	Denote $M(A)=[\alpha_{ij}]$ where $\alpha_{ii}=a_{ii}$ for all $i$ and $\alpha_{ij}=-|a_{ij}|$ for all $i \ne j$ and $\rho(A)=\max\{|\lambda|: \lambda\; \text{is an eigenvalue of}\; A \}$.  The following are all equivalent when $A \in \mathbb{S}^n$.
	\begin{enumerate}[\upshape (i)]
		\item $A$ is SDD;
		\item $M(A)$ is positive semi-definite.
	\end{enumerate} 
\end{lemma}

\begin{proof}[Proof of Lemma~\ref{lemma: 2oc=sdd}]
	We let $M(A)=[\alpha_{ij}] \in \mathbb{S}^{d}$ where $\alpha_{ii}=a_{ii}$  for $i=1,\dots,d$, and $\alpha_{ij}=-|a_{ij}|$ for all $i  \ne j$; this is often called the comparison matrix \cite{BP} of $A$. To show that  $\SDD{d}\supset(\psd{d})_2$, suppose $A\in (\psd{d})_2$. Then $A= \sum_{i<j} M^{ij}$. Since $M(A)=\sum_{i<j} M(M^{ij})$ with $M(M^{ij})\in \psd{d}$, $M(A)$ belongs to both $(\psd{d})_2$ and $\psd{d}$. It follows from Lemma~\ref{lemma: sdd results} that $A\in \SDD{d}$.
	
	Now suppose $A\in \SDD{d}$. There exists $d=(d_1,\dots,d_n)>0$  such that $d_i a_{ii} \geq \sum_{j \ne i}d_j |a_{ij}|$ for each $i$, which allows us to define $M^{ij}$ by
	\[
	m^{ij}_{ij} = a_{ij}, \quad m^{ij}_{ii}= \frac{d_j}{d_i}a_{ij}, \quad m^{ij}_{jj}=\frac{d_i}{d_j}a_{ij}.
	\]
	We may then increase the values of $m^{ij}_{ii}$ and $m^{ij}_{jj}$ appropriately so that they sum up to the respective diagonal entries of $A$. This shows that $\SDD{d}\subset(\psd{d})_2$.
\end{proof} 
For general $J\subseteq \comb{[d]}{2}$, the equality in Lemma \ref{lemma: 2oc=sdd} does not hold, i.e., $(\psd{d})_2(J)  \ne \SDD{d}$  for general $J$.

The $k$OCP in this case is actually very interesting. The $1$OCP is simply Linear Program (LP) since  $(\psd{d})_1 = \diag(\mathbb{R}_+) \cong \mathbb{R}_+$, the $2$OCP in this case is SDD program. We show in the following theorem that SDD program is the same as Second Order Cone Program (SOCP): 
\begin{equation}\label{socp}
\begin{aligned} 
& \text{minimize} & & a^\tp x  \\
& \text{subject to} &  &\|A_ix+b_i\|_2\leq c_i^\tp x+d_i, \quad i =1,\dots, q, \\  
& & &  Bx = e.
\end{aligned}
\end{equation}
\begin{theorem}\label{theorem: equivalence socp and SDD}
	SDD program is equivalent to SOCP, i.e., SOCP can be casted into SDD Program and vice versa. 
\end{theorem}
\begin{proof}
The fact that SDD program can be optimized using SOCP has been shown in \cite[Theorem 10]{ahmadi2017dsos}, which is just an easy consequence of Lemma \ref{lemma: 2oc=sdd}. We are only left to show the other direction. We will show one can transform a SOCP to the inequality form of SDD program. The equivalence between inequality form and standard form of SDD program follows from Theorem \ref{thm:equivalencesdfineq}.

Our only difficulty is to transform a SOC constraint,
\[\|A_ix+b_i\|_2\leq c_i^\tp x+d_i,\]
 to a SDD constraint.  We know
 \[\|A_ix+b_i\|_2\leq c_i^\tp x+d_i \iff \begin{bmatrix}
 (c_i^\tp x + d)I & A_ix +b_i\\
 (A_ix +b_i)^\tp  & c_i^\tp  x+d\\
 \end{bmatrix} \in \mathbb{S}_+^n\iff \begin{bmatrix}
 (c_i^\tp x + d)I & -|A_ix +b_i|\\
 -|(A_ix +b_i)^\tp | & c_i^\tp  x+d\\
 \end{bmatrix} \in \mathbb{S}_+^n \]
 for appropriate $n$ by the Schur complement condition for positive semi-definiteness, i.e., \[
 X= \begin{bmatrix}
 A & B \\ B^\tp  & C\\
 \end{bmatrix} \in \mathbb{S}^n_+ \iff A\in \mathbb{S}^m_+, C-B^\tp A^{-1}B\in \mathbb{S}^{h},
 \]
 where $m,h$ are number of rows of $A$ and $C$. Now using Lemma \ref{lemma: sdd results}, we see \[\begin{bmatrix}
 (c_i^\tp x + d)I & -|A_ix +b_i|\\
 -|(A_ix +b_i)^\tp | & c_i^\tp  x+d\\
 \end{bmatrix} \in \mathbb{S}_+^n \iff \begin{bmatrix}
 (c_i^\tp x + d)I & A_ix +b_i\\
 (A_ix +b_i)^\tp  & c_i^\tp  x+d\\
 \end{bmatrix} \in (\psd{n})_2.\] The last equation is a linear $(\psd{n})_2$ constraint and we see SOCP can be transformed to SDD program and so the two are equivalent.
\end{proof}
Thus we have shown that the intermediate program between LP, SOCP and SDP are $k$OCP and 
\begin{itemize}
	\item $1$OCP  = LP,
	\item $2$OCP = SOCP,
	\item $d$OCP = SDP,
	\item $k$OCP for $k=3,\dots, d-1$ are  intermediate programs: \[
	\begin{aligned} 
	&  \text{minimize} & & \tr(A_0X) \\
	& \text{subject to}\;& & \tr(A_iX)=b_i,\quad i=1,\dots,p \\
	& & & X\in (\psd{d})_k^*,
	\end{aligned}
	\]
	where $ (\psd{d})_k^* = \{ M : M=\sum_{ (i_1,\dots, i_k)\in {\comb{[d]}{k}}} \Aug{i_1 \cdots i_k}^d (M_{i_1 \cdots i_k}),\; M_{i_1 \cdots i_k}\in \psd{k}\}.$
\end{itemize}

The elements in higher order cone $(\psd{d})_k$ with $k\geq 3$ turns out to be known as factor-width $k$ matrices \cite{boman2005factor}. The corresponding program has being introduced in \cite{PP} before. 

The dual cones are : 
\[ ((\psd{d})_k)^* = \{A \in \mathbb{S}^{d} : \text{for all}\, (i_1,\dots,i_k)\in J,\,\Trn{i_1 \cdots i_k} (A) \in\psd{k}\} .\]
In the case of semidefinite cone, the nested inclusion for higher order cones $\{(\psd{d})_{k}\}_{k=1}^d$ and its dual cone series $\{((\psd{d})_k)^*\} _{k=1}^d$ are strict as shown in the following lemma. \label{key}
\begin{lemma}
	We have \[(\psd{d})_1\subsetneq(\psd{d})_2\subsetneq \dots \subsetneq (\psd{d})_d=\psd{d}\] and  \[((\psd{d})_1)^*\supsetneq((\psd{d})_2)^*\supsetneq \dots \supsetneq ((\psd{d}))^*_d=\psd{d}.\]
\end{lemma}
\begin{proof} Both inclusion are easy consequences of first and second item of Proposition \ref{proposition: koc}. We now prove the inclusion is strict. 
We first prove that the strict inclusion for the dual cones.  Denote $\mathbf{1}_d = (\underbrace{1, \dots,1}_{d \text{ copies}})^\tp $ and $I_d$ be the identity matrix in $\mathbb{S}^d$. The matrix 
\[
\begin{bmatrix} 

\sqrt{k-1} & \mathbf{1}^\tp_{d-1} \\
\mathbf{1}_{d-1}^\tp  & \sqrt{k-1}I_{d-1},\\
\end{bmatrix}
\]
is always in $((\psd{d})_k)^*$ but not in $((\psd{d})_{k+1})^*$.
 
Since $((\psd{d})_k)^*  = \cap_{(i_1,\dots,i_k) \in {\comb{[d]}{k}}}K_{i_1 \cdots i_k}$ where \[ K_{i_1 \cdots i_k} = \{A\in \mathbb{S}^{d} : \Trn{i_1 \cdots i_k} (A) \in\psd{k} \},\] $(K_{i_1 \cdots i_k})^* = \Aug{i_1 \cdots i_k} (\psd{k})$ and the identity matrix $I\in \interior(K_{i_1 \cdots i_k})$ for all $(i_1,\dots,i_k) \in {\comb{[d]}{k}}$, the Krein-Rutman Theorem \cite[Corollary 3.3.13]{borwein2010convex} implies that 
\[((\psd{d})_k)^{**} = \sum_{i_1 \cdots i_k} \Aug{i_1 \cdots i_k} (\psd{k}) = (\psd{d})_k .\]
Thus strict inclusion in the dual cones implies the strict inclusion in the cones $(\psd{d})_k$. The equality $ (\psd{d})_d=\psd{d}= ((\psd{d})_d)^*$ is because $\psd{d}$ is self-dual. 
\end{proof}
So far we have mostly dealing with index set $J_k = \comb{[d]}{k}$. By changing the index $J_k$ of the $k$th order cone, we obtain new cones and new conic program. In real problems, the choice of the subset $J_k$ of $\comb{[d]}{k}$ represents some prior knowledge of the problem. The corresponding higher order cone  and dual higher order cone prorgam can enojoy less computational budget because of the smaller size of $J_k$. This has been explored in the literature of chordal structure of SDP \cite{waki2006sums,de2010exploiting}. 

\section{Sum-of-squares cone}

A real coefficient polynomial $p(x)$ is a sum-of-square ($\SOS$) if it can be written as $p(x)= \sum_{i=1}^mq_i^2(x)$ for some polynomial $q_i$. It is clear that the set of sum of square polynomials form a convex cone. 

It is well-known that a polynomial  with $n$ variable and degree $2d$ is a sum of square if and only if there exists a positive semidefinite symmetric $A$ such that 
\[
p(x)=m(x)^\tp Am(x)
\]
where $m(x)$ is the vector of all monomials (so in total ${n+d\choose d}$ tuples) that have degree less than or equal to $d$ \cite{parrilo2000structured}. Due to this equivalence and our previous discussion on $k$OCP induced by $\psd{k}$, we define the following $k\ddsos$. 

\begin{definition}
Let $i_1,\dots,i_k \in \{ 1,2,\dots, {n+d\choose d}\}$.  A polynomial $p$ is $k$th-diagonally-dominant-sum-of-squares ($k\ddsos$) if it can be written as 
		\[
		p= \sum_{i_1 \cdots i_k} \sum_{j}\biggl (\sum_{l=1}^k \alpha^{i_1 \cdots i_k}_{ji_l}m_{i_l} \biggr)^2 
		\]
		for some monomials $m_{i_l}$ and some constants $ \alpha^{i_1 \cdots i_k}_{ji_l} \in \mathbb{R}$.
\end{definition}
 It directly follows from the definition that a polynomial $p$ (with $n$ variable and $2d$ degree) is SOS if and only if it is ${n+d\choose d} \ddsos$. The cases $k=1,2$ has been explored intensively in \cite{ahmadi2017dsos} under the name DSOS and SDSOS. 
 
 In the definition, we did not require $i_1<\dots<i_k$ as we did in defining $k$OC. We show in the following lemma that this requirement is not necessary. 
 
 \begin{lemma}\label{lm2}
 	Suppose  the monomials having $n$ variables with degree less than or equal to $d$ are indexed by $\{1,2,\dots,{n+d\choose d}\}$ according to some order. A polynomial $p$ with degree $2d$ , $n$ variables is $k\ddsos$ if and only if it can be written as 
 	\[
 	p= \sum_{(i_1,\dots, i_k)\in{ \comb{[n+d]}{ d}}} \sum_{j=1}^k \biggl(\sum_{l=1}^k \alpha^{i_1 \dots i_k}_{ji_l}m_{i_l} \biggr)^2, 
 	\]
 	where $(m_{i_l})_{l=1}^k$ are different for different $(i_l)_{l=1}^k$. 
 \end{lemma}

\begin{proof}
	It is easy to see a polynomial can be written in the above form is a $k\ddsos$.
	
	Now suppose $p$ is a $k\ddsos$, by rearrange the brackets and adding $0$ terms if necessary, we could write $p$ in the form 
	\[
	p= \sum_{(i_1,\dots, i_k)\in\comb{[n+d]}{ d}} \sum_{j} \biggl(\sum_{l=1}^k \alpha^{i_1 \cdots i_k}_{ji_l}m_{i_l}\biggr )^2, 
	\]
	where  $m_{i_l}$s are different when $i_l$s are not equal. So the thing left to do is to make sure there are $k$ brackets in the second sum, i.e., the sum over $j$. Since 
	\[
	(\sum_{l=1}^k \alpha^{i_1 \cdots i_k}_{ji_l}m_{i_l} )^2 =m_{i_1 \cdots i_k}^\tp  (\alpha^{i_1 \cdots i_k}_j)^\tp \alpha^{i_1 \cdots i_k}_j m_{i_1 \cdots i_k},
	\]
	where $m_{i_1 \cdots i_k} = (m_{i_1},\dots, m_{i_k}),\alpha^{i_1 \cdots i_k}_j=(\alpha^{i_1\dots  i_k}_{ji_1},\dots, \alpha^{i_1 \cdots i_k}_{ji_k}).$ The sum 
	$\sum_j (\alpha^{i_1 \cdots i_k}_j)^\tp \alpha^{i_1 \cdots i_k}_j $ is still a non-negative definite matrix and thus has a Cholesky decomposition,i.e.,   $\sum_j (\alpha^{i_1 \cdots i_k}_j)^\tp \alpha^{i_1 \cdots i_k}_j = D^\tp D$. This means 
	\[
	p= \sum_{i_1 \cdots i_k}( m_{i_1 \cdots i_k}D)^\tp  D(m_{i_1 \cdots i_k}),
	\]
	which shows there can be exactly $k$ brakets in the second sum. 
\end{proof}

The following theorem connects our $k\ddsos$ polynomial with our $k$th order cone induced by $\psd{k}$.

\begin{theorem}\label{theorem: sos}
	A polynomial $p$ of degree $2d$ with $n$ variables is $k\ddsos$ if and only if it admits a representation as $p(x)=m^\tp (x)Am(x)$, where $m(x)$ is the standard monomial vector of degree $d$ (so in total ${n+d\choose d}$ tuples with different entries), and $A \in (\psd{h})_k$ for some $h\leq {n+d\choose d}$.
\end{theorem}

\begin{proof}
	If $p$ admits a representation
		\[
		p =m^\tp Am,
		\]
		where $A\in (\psd{h})_k$ for some $h$ and $m$ is the vector of all monomials with degree less than $d$. Since $A\in (\psd{h})_k$, $A$ has the decomposition $A = \sum_{(i_1,\dots, i_k)\in \comb {[n+d]}{d}}M^{i_1 \cdots i_k}$. $M^{i_1 \cdots i_k}$ are zero except for those $(i,j), i,j\in \{i_1,\dots,i_k\}$ entries. $\Trn{i_1 \cdots i_k}^h(M ^{i_1 \cdots i_k})$ are positive semi-definite and thus has the Cholesky decomposition $\Trn{i_1 \cdots i_k}^h(M^{i_1 \cdots i_k}) = N_{i_1 \cdots i_k}^\tp  N_{i_1 \cdots i_k}$. Thus, we have 
		\begin{align*}
		p & = m^\tp Am \\ 
		& = \sum_{i_1 \cdots i_k} m^\tp  M^{i_1 \cdots i_k}m \\
		& = \sum_{i_1 \cdots i_k} \begin{bmatrix}
		m_{i_1} & \dots & m_{i_k} 
		\end{bmatrix} \Trn{i_1 \cdots i_k}^d(M^{i_1 \cdots i_k})\begin{bmatrix}
		m_{i_1} \\ \vdots \\ m_{i_k}
		\end{bmatrix}\\
		&  =\sum_{i_1 \cdots i_k}  (\begin{bmatrix}
		m_{i_1} & \dots & m_{i_k} 
		\end{bmatrix} N_{i_1 \cdots i_k}^\tp ) \Biggl(N_{i_1 \cdots i_k}\begin{bmatrix}
		m_{i_1} \\ \vdots \\ m_{i_k}\end{bmatrix}\Biggr)
		\end{align*}
The last expression shows that $p$ is a $k\ddsos$.
		
Now if $p$ is a $k\ddsos$, as shown in lemma \ref{lm2}, we could write 
		\[
		p =  \sum_{i_1 \cdots i_k} \sum_{j=1}^k \biggl(\sum_{l=1}^k \alpha^{i_1 \cdots i_k}_{ji_l}m_{i_l} \biggr)^2,
		\]
		where $m_{i_l}$ are different for different $i_l$.
This gives our
		\[N_{i_1 \cdots i_k} = \begin{bmatrix}
		\alpha_{1i_1}^{i_1 \cdots i_k} & \dots & \alpha_{1i_k}^{i_1 \cdots i_k} \\
		\alpha_{2i_1} ^{i_1 \cdots i_k}&  \dots  & \alpha_{2i_k}^{i_1 \cdots i_k}\\
		\vdots & \dots & \vdots \\
		\alpha_{ki_1}^{i_1 \cdots i_k} & \dots & \alpha_{ki_k}^{i_1 \cdots i_k}
		\end{bmatrix}.
		\]
We then can construct $M^{i_1 \cdots i_k}$ and $A$.
\end{proof}
We define the corresponding $k\ddsos$ program here.
\begin{definition} 
	Denote the cone of $k\ddsos$ with degree $2d$ and $n$ variables as $k\SOS_{n,d}$. We call the following optimization $k\ddsos$ programming.
	\begin{equation}\label{program: kDSOS}
	\begin{aligned} 
	& \underset{u\in \mathbb{R}^l}{\text{minimize}} & & r^\tp u  \\
	& \text{subject to} &  &r_{0,t}+ r_{1,t}(x)u_1+\dots +r_{s_t,t}(x)u_{s_t} \in k\SOS_{n_t,d_t}, t = 1,2,\dots, N,\\  
	& & &  
	\end{aligned}
	\end{equation}
	where $r(x)$s are given polynomials and $n_t,d_t$ depends on $r(x)$. $n_t$ is the total number of variables of $r(x)$ in the same inequality. $d_t$ is half the highest degree of $r(x)$ in the same inequality.
\end{definition}

To link to our previous discussion of $k$OCP induced by $\psd{k}$, we show that these two programs are equivalent. 
\begin{theorem}
	$k\ddsos$ programming is equivalent to $(\psd{d})_k$ cone programming ($k$OCP induced by $\psd{k}$).
\end{theorem} 
 \begin{proof}
	We first show how to reduce $(\psd{d})_k$ cone program to $k\ddsos$ program: 
		
		We may suppose $(\psd{d})_k$ program is in its standard form, i.e., the form in \eqref{program:kocpstandard} (the equivalence between standard form and inequality form for $(\psd{d})_k$ can be proved via standard techniques). To avoid confusion, suppose $U$ is the variable matrix  in $(\psd{d})_k$ cone program. 
		
		Then our $r$ in $k\ddsos$ program \eqref{program: kDSOS} is just $\text{vec}(A_0)$. The linear equality can be  incorporated into a $k\ddsos$ inequality by let $r$s in the inequality in \eqref{program: kDSOS} be constant and matches $(D_i)_{jk},-(D_i)_{jk}$ as the following 
		\[
		\tr(D_iU)=f_i \quad \iff \quad -f_i+\sum_{jk} (D_{i})_{jk}u_{ij}   \in k\SOS_{1,1} \text{ and } f_i+\sum_{jk} -(D_{i})_{jk}u_{ij} \in k\SOS_{1,1}.
		\]
		
		The condition $U\in( \psd{d})_k$ is the same as 
		\[
		\sum_{1\leq i,j\leq d} x_i x_j u_{ij} \in k\SOS_{d,1}
		\]
		by Theorem~\ref{theorem: sos}.
		
	Next we show how to reduce $k\ddsos$ program to $(\psd{d})_k$ cone program in its inequality form.
		The objective is the same for both program.
		
		The constraint $p_t(x)=r_{0,t}(x)+ r_{1,t}(x)u_1+\dots +r_{s_t,t}(x)u_k \in k\SOS_{n_t,d_t}$ is the same as there is one $A=[a_{ij}]_{ij}\in (\psd{h})_k$ for some $h$ such that $p_t(x)= m^\tp Am$. Thus  
		\[
		p_t(x)=r_{0,t}(x)+ r_{1,t}(x)u_1+\dots r_{s_t,t}(x)u_t \in k\SOS_{n_t,d_t}
		\]
		if and only if there exists
		\[
		A \in (\psd{n_t+d_t\choose d_t})_k,\quad \text{and linear constrants on}\,a_{ij},u_{i},
		\]
		where the linear constrants come from matching coefficients of $p_t(x)=m^\tp Am =r_{0,t}+ r_{1,t}(x)u_1+\dots+ r_{s_t,t}(x)u_t $. The condition $A\in (\psd{n_t+d_t\choose d_t})_k$ is a $k$OC constraint and we could add  variable $a_{ij}$ to $k$OCP. This shows the other direction.

\end{proof}

\section{Completely positive cone and copostive cone}

Recall the following definition of completely positive matrices and copositive matrices: 
\begin{itemize}
\item The set of copositive matrices with dimension $d$, $\cp{d}$:
\[
\cp{d} :\,=\{M\in \mathbb{S}^d  : x^\tp Mx \geq 0  \text{ for all } x\in \mathbb{R}^d_+\}.
\]
\item The set of complete positive matrices with dimension $d$, $\cpp{d}$:
\[
\cpp{d}\coloneqq\{ B^\tp B : B\in \mathbb{R}^{m\times d}_+,\; m \text{ is an integer} \}.
\]
\end{itemize}

These two cones satisfy the \ep throughly by verifying the definition directly. The corresponding copositive programming and copositive programming gives a lot modeling power in combinatorics and nonconvex problems \cite{dur2010copositive,burer2015gentle}. However, these programs are NP-hard to solve in general. 

Using the construction of $k$OCP induced by $\cp{k}$ or $\cpp{k}$, for $k=1,2,3,4$, we are able to solve the the inner approximation of copositive programming and copositive programming. The case $k=2$ of $\cpp{k}$ has been explored in \cite{bostanabad2018inner}. 
\begin{theorem}\label{thm: cpcpptohocp}
	$2,3,4$-OCP with index set $J$ induced by $\cp{k}$ or $\cpp{k}$ can be casted into $2,3,4$-OCP induced by $\psd{k}$. 
\end{theorem}

The theorem is mainly due to the following lemma:
\begin{lemma}\label{lemma: cppcppsd} \cite{maxfield1962matrix} 
	Denote $\nn{k} =  (\mathbb{R}^{k\times k}_+)\cap \mathbb{S}^k$, we have for $k=1,2,3,4$
\[\cp{k} = \psd{k} + \nn{k} ,\quad \cpp{k}  = \psd{k} \cap \nn{k}.\]
\end{lemma}
\begin{proof}[Proof of Theorem \ref{thm: cpcpptohocp}]
	We may suppose $i$-OCP with index set $J$ induced by $\cp{k}$ or $\cpp{k}$ are in its standard form \eqref{program:kocpstandard} where $i=2,3,4$. The case of inequality form is similar. 
	
	The constraint $X\in \cp{d}_i$ is the same as 
	\[X\in \cp{d}_i \iff X = \sum_{\{j_1,\dots,j_i\}\in J,j_1<\dots <j_i} \Aug{j_1\dots j_i}(M_{j_1\dots j_i}) \quad \text{and}\quad M_{j_1\dots j_i}\in \cp{i}.\]
	
	Since $M_{j_1\dots j_i}\in \cp{i}$ if and only if $M_{j_1\dots j_i} = S_{j_1\dots j_i} + N_{j_1\dots j_i}$ for some $S_{j_1\dots j_i} \in \psd{i}$, $N_{j_1\dots j_i}\in \nn{i}$ by Lemma \ref{lemma: cppcppsd} . We see $i$-OCP with index set $J$ induced by $\cp{k}$  can be casted into $i$OCP induced by $\psd{k}$ (the constraint $S_{j_1\dots j_i} \in \psd{i}$ can be casted into $(\psd{d})_k$ constraint by setting $d = k$ and the nonnegative constraint can be handled via $\diag(x)\in (\psd{d})_k\iff x \geq 0$).
	
	For $i$-OCP with index set $J$ induced by $\cpp{k}$. We note that 	\[X\in \cp{d}_i \iff X = \sum_{\{j_1,\dots,j_i\}\in J,j_1<\dots <j_i} \Aug{j_1\dots j_i}(M_{j_1\dots j_i}) \quad \text{and}\quad M_{j_1\dots j_i}\in \cpp{i}.\]
	
	Since $M_{j_1\dots j_i}\in \cpp{i}$ if and only if $M_{j_1\dots j_i} \in \psd{i}$ and $M_{j_1\dots j_i}\in \nn{i}$ by Lemma \ref{lemma: cppcppsd} . We see $i$-OCP with index set $J$ induced by $\cp{k}$  can be casted into $i$OCP induced by $\psd{k}$.
\end{proof}
By adjusting the set $J$ and an $U\in \mathbb{R}^{d \times d }$ , we may consider solving 
\[
\begin{aligned} 
&  \text{minimize} & & \tr(A_0X) \\
& \text{subject to}\;& & \tr(A_iX)=b_i,\quad i=1,\dots,p \\
& & & X\in U\cp{d}_k(J)U^\tp .
\end{aligned}
\]
and 
\[
\begin{aligned} 
&  \text{minimize} & & \tr(A_0X) \\
& \text{subject to}\;& & \tr(A_iX)=b_i,\quad i=1,\dots,p \\
& & & X\in U\cpp{d}_k(J)U^\tp .
\end{aligned}
\]
This formulation gives us more modeling power and can also be casted into $k$OCP induced by $\psd{k}$ for $k=2,3,4$.
\section{Symmetric cones}
\subsection{Positive semidefinite matrices in $\mathbb{R}^{d\times d},\mathbb{C}^{d\times d}, \mathbb{H}^{d \times d}$ and $\mathbb{O}^{d\times d}$}

Let us first recall the five irreducible symmetric cones\footnote{A cone is symmetric if it is self-dual and its autonomous group acts transitively on it. A symmetric cone is irreducible means it cannot be written as a Cartesian product of other symmetric cones}: 
\begin{enumerate}[\upshape (i)]
	\item Symmetric real positive semidefinite  matrices in $\mathbb{S}^{d}$
	\item Hermitian complex positive semidefinite  matrices in $\mathbb{C}^{d \times d}$
	\item Hermitian quaternion positive semidefinite matrices  in $\mathbb{H}^{d \times d}$
	\item Hermitian octonian positive semidefinite  $\mathbb{O}^{3\times 3}$ matrices
	\item Second order cone in $\mathbb{R}^{d+1}$: $\SOC^{d+1} = \{ (t,x)\mid \twonorm{x}\leq t, x\in \mathbb{R}^{d},t\in \mathbb{R}\}$.
\end{enumerate}

For the first three cones, they satisfy the \ep throughly as they are all of the form 
\[ \{A \in \mathbb{S}^d : x^*Ax \geq 0, \,\text{for all}\, x\in \mathbb{F}^{d}\},\] where $\mathbb{F} = \mathbb{R}, \mathbb{C}$ or $\mathbb{H}$.  

Let \[\Ho{d}_+ = \{A \in \mathbb{O}^{d \times d} : A= A^*, x^*Ax \geq 0, \,\text{for all}\, x\in \mathbb{O}^{d}\}.\] We may consider the series $\{\Ho{k}\}_{k=1}^\infty$ so that  the cone of Hermitian octonian positive semidefinite  $\mathbb{O}^{3\times 3}$ matrices is a member of it. The series satisfies the \ep throughly. 
\subsection{Second Order Cone} We need to first transform the second order cone into the space of symmetric matrices. This can be done through: 
\begin{equation}
\begin{aligned}\label{def: sock}
\{A : A = \diag(t,x), \, \text{for some}\,(t,x)\in \SOC^{k}\}
\end{aligned}
\end{equation}
We abuse the notation and call the above set as $\SOC^{k}$ as well. Moreover, we define $\SOC^1 = \mathbb{R}_+$. The index map for $\{\SOC^{k}\}_{k=1}^\infty$ is then \[I_{\SOC} (d,k)= \{s  : s = (1,i_1,\dots, i_{k-1}), 2\leq i_1<\dots <i_{k-1}\leq d\}\] for $k\geq 2$ and is simply $\{1\}$ if $k=1$. 
It can be easily  verified that $\{\SOC^{k}\}_{k=1}^\infty$ satisfies the \ep with index map $I(d,k)_{\SOC}$.

We can avoid lifting the second order cone to $d \times d$ matrices.  First, we define $\Trn{i_1 \cdots i_k} ^{\mathbb{R}^d}: \mathbb{R}^{d} \rightarrow \mathbb{R}^{k}$, $\Aug{i_1 \cdots i_k}^{\mathbb{R}^d}:\mathbb{R}^{k}\rightarrow \mathbb{R}^{d}$ for every $(i_1,\dots,i_k)\in{\comb{d}{k}}$ such that \[ \Trn{i_1 \cdots i_k} ^{\mathbb{R}^d}(x) = (x_{i_1},\dots,x_{i_k}), \quad [\Aug{i_1 \cdots i_k}^{\mathbb{R}^d}(y)]_{i}= \begin{cases}
y_j  & i = i_j \,\text{for some}\,i_j\\
0 & \text{otherwise}
\end{cases}.\]

The $k$th higher order cone of $\SOC^{d}$ is then 
\[
\SOC^{d}_{k} = \{ x\in \mathbb{R}^{d}  : x = \sum_{(1,i_1,\dots ,i_{k-1})\in {I}_{\SOC}(d,k)} \Aug{1i_1 \cdots i_{k-1}}(x_{1i_1\dots i_{k-1}}), x_{1i_1 \cdots i_{k-1}} \in \SOC^{k} \}.
\]

The following Lemma shows the nest inclusion of $\{\SOC^{d}_k\}_{k=1}^d$ is strict.

\begin{lemma}
	We have \[ \SOC^{d}_1\subsetneq \SOC^{d} _2\subsetneq\SOC^{d}_3\subsetneq \dots  \subsetneq \SOC^{d}_d = \SOC^{d}.\]
\end{lemma}
\begin{proof}
	The inclusion follows easily from the Nested Cone property in Proposition \ref{proposition: koc}. We now prove the inclusion  is actually strict.
	First, we consider the dual cones \[(\SOC^{d}_k)^*= \{x\in \mathbb{R}^{d} : \Trn{1i_1 \cdots i_{k-1}} ^{\mathbb{R}^d}(x) \in \SOC^{k} \,\text{for all}\,(1,i_1,\dots ,i_{k-1})\in {I}_{\SOC}(d,k)\}.\] An application of first and second item of \ref{proposition: koc} tells us that  $$(\SOC^{d}_1)^*\supseteq (\SOC^{d} _2)^*\supseteq(\SOC^{d}_3)^*\supseteq \dots  (\supseteq \SOC^{d}_d )^*= (\SOC^{d})^*.$$
	Consider $(\sqrt{k-1},\mathbf{1}_{d-1})$, where $\mathbf{1}$ is a all one vector with length $d-1$. This vector belongs $(\SOC^{d}_k)^*$ but not $(\SOC^{d}_{k+1})^*$. Thus the inclusion in the dual cones is strict. 
	
	Since  $(\SOC^{d}_k)^* = \cap_{(1,i_1,\dots ,i_{k-1})\in {I}_{\SOC}(d,k)} K_{i_1 \cdots i_{k-1}}$ where $K_{i_1 \cdots i_{k-1}} = \{x\in \mathbb{R}^{d} : \Trn{1i_1 \cdots i_{k-1}} ^{\mathbb{R}^d}(x) \in \SOC^{k} \}$, and $(\sqrt{d+1}, \mathbf{1}_{d-1})\in \interior(K_{i_1 \cdots i_{k-1}})$, by the Krein-Rutman Theorem \cite[Corollary 3.3.13]{borwein2010convex}, we have
	\[(\SOC^{d}_k)^{**}= \sum_{(1,i_1,\dots ,i_{k-1})\in {I}_{\SOC}(d,k)}\Aug{1i_1 \cdots i_{k-1}}^{\mathbb{R}^{d}}(\SOC^{k}) =\SOC^{d}_k .\]
	Thus the strict inclusion in the dual cone implies that strict inclusion in $\{\SOC^d_k\}_{k=1}^d$. 
\end{proof}
\section{Norm Cones}
The \ep property is also satisfied by a large class of norm cones. Specifically, the property we need is the following. 
\begin{definition}\label{def: consistenceMonotonic} 
Suppose a norm $\|\cdot\|$ is defined on $\mathbb{S}^d$ (or $\diag(\mathbb{R}^d)$) for all $d$. For any $1\leq k\leq d$, $X\in \mathbb{S}^d$ (or $\diag(\mathbb{R}^d)$) and any $(i_1,\dots, i_k)\in {\comb{[d]}{k}}$, it is 
\begin{enumerate}[\upshape (i)]
	\item \emph{consistent} if $\|\Trn{i_1\cdots i_k}(X)\| = \| \Aug{i_1\cdots i_k}(\Trn{i_1\cdots i_k}(X))\|$;
	\item \emph{monotonic} if  $\|\Trn{i_1\cdots i_k}(X)\| \leq \|X\|$.
\end{enumerate}
\end{definition}

Norms satisfied the consistency and monotonicity are abundant, for example, 
\begin{enumerate}[\upshape (a)]
	\item All $\ell_p$ norms on $\mathbb{R}^d$: $\|x\|_p = (\sum_{i=1}^d|x_i|^p)^{\frac{1}{p}}$ for any $x\in \mathbb{R}^d,p\geq 1$. 
	\item All Schatten norm on $\mathbb{S}^d$ with underlying field being $\mathbb{R}$ or $\mathbb{C}$: $\|X \|_p = (\sum_{i=1}^d|\sigma_i(X)|^p)^{\frac{1}{p}}$ for all $X\in \mathbb{S}^d$ where $\sigma_i(X)$ is the $i$th  largest singular value of $X$. The monotonicity  is due to Cauchy's interlace theorem. 
	\item All Ky-Fan $k$ norm on $\mathbb{S}^d$ with underlying field being $\mathbb{R}$ or $\mathbb{C}$: $\|X \|_{\text{KF}_k}= \sum_{i=1}^k\sigma_i(X)$ for all $X\in \mathbb{S}^d$ and $\sigma_i =0$ for $i>d$. The monotonicity  is also due to Cauchy's interlace theorem. 
	\item The operator norm of matrix induced by $\ell_p$, $\ell_q$ vector norms: $\|A\|_{p,q} = \sup_{\|x\|_p=1} \|Ax\|_{q}$ for any $1\leq p,q\leq \infty$.
\end{enumerate}

 In fact, these two properties turns out to be the characterization of norms so that its corresponding norm cones having \ep as $\SOC$. This fact is shown by the following theorem.

\begin{theorem}[Characterization of Norm Cones satisfying \ep as $\SOC$]
For a norm $\|\cdot\|$ defined on $\mathbb{S}^d$ (or $\diag(\mathbb{R}^d)$ for all $d\geq 1$, let the norm cone in $\mathbb{S}^{d+1}$ (or $\diag(\mathbb{R}^{d+1})$ be
\[
N_{\|\cdot\|}^{d+1} =\{ \diag(t,X)\mid \|X\|\leq t, X\in \mathbb{L}^d,t \in \mathbb{R}\},
\]
and $N_{\|\cdot\|}^{1} = \mathbb{R}_+$ where $\mathbb{L}^d= \mathbb{S}^d$ or $\diag(\mathbb{R}^d)$. If the norm is consistent and monotonic, then the series of norm cones $\{N_{\|\cdot\|}^k\}_{k=1}^\infty$ satisfies the \ep with index map $I_{\SOC}(d,k)$. The converse is also true. 
\end{theorem}

\begin{proof} We prove the case of $\mathbb{S}^d$. The case of $\diag(\mathbb{R}^d)$ follows exactly the same line. 
	
	We first show that consistency with monotonicity implies that $N_{\|\cdot\|}^d$ satisfies the \ep with index map $I_{\SOC}$. 
	For any $1\leq k\leq d$, $(1,i_1,\dots ,i_{k-1})\in {I}_{\SOC}(d,k)$, $\diag(t,X)\in N_{\|\cdot\|}^{d+1}$, and $\diag(s,Z) \in N_{\|\cdot \|}^{k+1}$, the consistency implies that 
	\[\Aug{1i_1\cdots i_k}(\diag(s,Z)) \in N_{\|\cdot\|}^{d+1},\] since 
	\[\Aug{1i_1\cdots i_k}(\diag(s,Z)) = \diag(s, \Aug{i_1\cdots i_k}(Z)), \quad \text{and} \quad \|\Aug{i_1\cdots i_k}(Z)\| = \|Z\|\leq s.\]
	The monotonicity implies that \[\Trn{1i_1\cdots i_k} (\diag(t,X)) \in N_{\|\cdot\|}^{k+1},\] since 	\[\Trn{1i_1\cdots i_k}(\diag(t,X)) = \diag(t, \Trn{i_1\cdots i_k}(X)), \quad \text{and} \quad \|\Trn{i_1\cdots i_k}(X)\| \leq  \|X\|\leq t.\] The case $k=0$ is trivial. 

	Next we show the \ep of $\{N_{\|\cdot\|}^k\}_{k=1}^\infty$ implies its consistency and monotonicity. 
	Due to the \ep of $\{N_{\|\cdot\|}^{k}\}_{k=2}^d$, we have for any  $1\leq k\leq d$, $(1,i_1,\dots ,i_{k-1})\in {I}_{\SOC}(d,k)$, $Z\in \mathbb{S}^k$, 
	\[ \diag(\|Z\|,Z) \in N_{\|\cdot\|}^{k+1} \implies \Aug{1i_1\cdots i_k}(\diag(\|Z\|,Z)) \in N_{\|\cdot\|}^{d+1} \implies \|\Aug{1i_1\cdots i_k}(Z)\|\leq \|Z\|. \]
	Now consider
\[
\diag(\|\Aug{1i_1\cdots i_k}(Z)\|,\Aug{1i_1\cdots i_k}(Z)) \in N_{\|\cdot\|}^{d+1} \implies \Trn{1i_1\cdots i_k}(\diag(\|\Aug{1i_1\cdots i_k}(Z)\|,\Aug{1i_1\cdots i_k}(Z)))\in  N_{\|\cdot\|}^{k+1}.
\]
But $\Trn{1i_1\cdots i_k}(\diag(\|\Aug{1i_1\cdots i_k}(Z)\|,\Aug{1i_1\cdots i_k}(Z))) = \diag(\|\Aug{1i_1\cdots i_k}(Z)\|,Z)$ and we have \[\|\Aug{1i_1\cdots i_k}(Z)\|\geq \|Z\|.\]
	This shows the consistency by taking $Z = \Trn{i_1\cdots i_k}(X)$.
	
	To prove monotonicity, we have for any $X\in \mathbb{S}^d$, \[\Trn{1i_1\cdots i_k} (\diag(t,X)) \in N_{\|\cdot\|}^{k+1} \implies  \|\Trn{i_1\cdots i_k}(X)\| \leq t\] and taking $t= \|X\|$ shows the monotonicity.
\end{proof}

Thus the norm cones of our previous mentioned four kinds of norm (1) $\ell_p$ norm on $\mathbb{R}^d$, (2) Schatten norm on $\mathbb{S}^d$ with underlying field being $\mathbb{R}$ or $\mathbb{C}$,  (3) Ky-Fan $k$ norm with underlying field being $\mathbb{R}$ or $\mathbb{C}$, and (4) operator norm induced by $p,q$ norms all satisfies the \ep with index map $I_{\SOC}$. This means our previous discussion on $\SOC$ is just a special case of norm cones with \ep. 

Here we give two more concrete examples of norm with consistency and monotonicity and studies its $k$th order cone. Let us first consider the $\ell_1$ norm in $\mathbb{R}^d$. As in the case of the second order cone, we don't need to lift the space to matrices. The second order cone induced by $N_{\ell_1}^{d+1}$ is 
\[
(N_{\ell_1}^{d+1})_2 = \{ (t,x) \mid (t,x) = (\sum_{i=1}^d t_1+\cdots t_d, x_1,\dots,x_d), \,\text{for all}\,i, \, t_i\geq |x_i|, \,t_i,x_i \in \mathbb{R}\}.
\]
which is simply $N_{\ell_1}^{d+1}$! Thus by first item of Proposition \ref{proposition: koc}, we know the $k$th order cone induced by $N_{\ell_1}^{d+1}$ is just itself $N_{\ell_1}^{d+1}$ for $k\geq 2$. We don't gain new cones from this construction except the trivial cone $(N_{\ell_1})^{d+1} = \mathbb{R}_+\times \{0_{d}\}$ where $0_d$ is the zero vector of length $d$. Note that this is not the case for the second order cone. 

Next we consider the nuclear norm $\|\cdot\|_*$: 
\[ \|X\|_* = \sum_{i=1}^d \sigma_i(X), \quad \text{for all}\quad X\in \mathbb{S}^d\] with underlying field being real or complex. The $(k+1)$th order cone induced by $N_{\|\cdot\|_*}^{d+1}$ is 
\[ (N_{\|\cdot\|_*}^{d+1})_{k+1} = \{\sum_{(1,i_1,\dots ,i_{k-1})\in {I}_{\SOC}(d,k)} \Aug{1i_1\cdots i_k}(\diag(t_{i_1\cdots i_k},X_{i_1\cdots i_k})) \mid \diag(t_{i_1\cdots i_k},X_{i_1\cdots i_k}) \in N_{\|\cdot \|_*}^{k+1}\}.\]
Since $(N_{\|\cdot\|_*}^{d+1})^* = N_{\|\cdot\|_{2}}^{d+1}$ where $\|\cdot\|_{2}$ is the operator two norm, we know from second item of Proposition \ref{proposition: koc}, the dual cone of 
		$(N_{\|\cdot\|_*}^{d+1})_{k+1}$ is 
		\[ ((N_{\|\cdot\|_*}^{d+1})_{k+1})^* = \{\diag(t,X) \mid \Trn{1i_1\cdots i_k} (\diag(t,X))\in N_{\|\cdot \|_2}^{k+1}, \,\text{for all} \, (1,i_1,\dots ,i_{k-1})\in {I}_{\SOC}(d,k)\}.\] 
Moreover, by an application of first and second items of Proposition \ref{proposition: koc}, we have \[ ((N_{\|\cdot\|_*}^{d+1})_{d+1})^*\subseteq ((N_{\|\cdot\|_*}^{d+1})_{d})^* \subseteq \cdots \subseteq  ((N_{\|\cdot\|_*}^{d+1})_{2})^*\subseteq  ((N_{\|\cdot\|_*}^{d+1})_{1})^*. \] 
By considering $\diag(k,I_d  + \mathbf{1}\mathbf{1}^\top)$ with $k=1,2,\dots, d$ where $I_d$ is the identity matrix in $\mathbb{S}^d$ and $\mathbf{1}$ is the all one vector, we find that \[ ((N_{\|\cdot\|_*}^{d+1})_{d+1})^*\subsetneq ((N_{\|\cdot\|_*}^{d+1})_{d})^* \subsetneq \cdots \subsetneq  ((N_{\|\cdot\|_*}^{d+1})_{2})^*\subsetneq  ((N_{\|\cdot\|_*}^{d+1})_{1})^*. \] Since $\diag(d+1,I_d) \in \interior( K_{i_1 \cdots i_{k}} )= \interior(\{\diag(t,X): \Trn{1i_1 \cdots i_{k}} (\diag(t,X) )\in N_{\|\cdot\|_{2}}^{k+1}\})$ and $((N_{\|\cdot\|_*}^{d+1})_{k+1})^* = \cap_{(1,i_1,\dots ,i_{k-1})\in {I}_{\SOC}(d,k) }K_{i_1 \cdots i_{k}} $, by the Krein-Rutman Theorem \cite[Corollary 3.3.13]{borwein2010convex}, we find that
		\[ (N_{\|\cdot\|_*}^{d+1})_{d+1}\supsetneq (N_{\|\cdot\|_*}^{d+1})_{d}\supsetneq \cdots \supsetneq  (N_{\|\cdot\|_*}^{d+1})_{2} \supsetneq  (N_{\|\cdot\|_*}^{d+1})_{1}. \] Thus, unlike the case of $\ell_1$ norm cone, we indeed obtain new cones here.

Finally, the $k$th order cone program induced by $(N_{\|\cdot\|}^{d+1})_{k}$ for monotonic and consistent norm $\|\cdot\|$ is 
\begin{equation*}
\begin{aligned} 
&  \text{minimize} & & \tr(A_0X) +a_0t \\
& \text{subject to}\;& & \tr(A_iX)+a_it=b_i,\quad i=1,\dots,p \\
& & & \diag(t,X)\in  (N_{\|\cdot\|}^{d+1})_k,
\end{aligned}
\end{equation*}
where $A_i\in \mathbb{S}^d$, $a_i,\,b_i\in \mathbb{R}$, for all $i=0,\dots, p$.

\section{KKT Condition and Self-Concordance}
\subsection{KKT Condition}
Here we list the KKT condition for our higher order cone program. The primal form of our $\cone{d}_k(J)$ program is 
\begin{equation}
\begin{aligned} \label{Optprogram: primal higher order}
& \text{minimize}& & \tr(A_0X) \\
& \text{subject to}\;& & \tr(A_iX)=b_i ,\quad i =1,\dots,p \\
& & & X\in \cone{d}_k(J)
\end{aligned}
\end{equation}

The dual of the above program is 
\begin{equation}
\begin{aligned} \label{Optprogram: dual higher order}
& \text{minimize}& & b^\tp  y \\
& \text{subject to}\;& & A_0 - \sum_{i=1}^p y_i \in (\cone{d}_k(J))^*.\\
\end{aligned}
\end{equation}

Let  $\optx,\opty$ be a primal and dual solution pair of the above programs. Also let $\optz = A_0 - \sum_{i=1}^pA_i\opty_i$. If strong duality holds: 
\[ \tr(A_0 \optx) = b^\tp  \opty,\] we have the KKT condition as 
\begin{equation} \label{eq: kkt}
\begin{aligned} 
\optx& \in \cone{d}_k(J)\\
\tr(A_i\optx )& =b_i, i=1,\dots,p\\
\optz  &\in (\cone{d}_k(J))^*,\\ 
\tr (\optz \optx) & =0,\\
A_0- \sum_{i=1}^p A_i\opty_i& =\optz .
\end{aligned}
\end{equation}

\subsection{Self-Concordance}
We assume the original cone $\cone{k}$ and its $k$th order cone $\cone{d}_k$ are proper and the underlying field is $\mathbb{R}$. The index set is $I(d,k)$. The assumption of properness on $k$-th order  is true for all previous mentioned examples in $\mathbb{S}^{d}$.

Recall the definition of self-concordance and a few propositions of it.  
\begin{definition}
	Let $K$ be a convex closed cone. A continuous function $f: K\to \mathbb{R}\cup\{+\infty\}$ is called a barrier function of $K$ if it satisfies 
	
	\[f(x)<\infty \quad \text{for every } x \in \interior(K) \quad f(x)=+\infty \quad \text{for every }x \in \partial K, \]
	where $K^{\circ}$ means taking the interior of $K$ and $\partial K$ means the boundary of $K$ induced by the usual topology in $\mathbb{R}^n$.
	
	A convex third order differentiable function $f(x)$ on $K$ is self-concordant if for every $x\in K^{\circ}$ and $h\in \mathbb{R}^n$ the univariate function $\phi (\alpha) = f(x+\alpha h)$ satisfies the property 
	\[|\phi'''(0)|\leq 2 |\phi''(0)|^{\frac{3}{2}} \quad . \] 
	
	A barrier $f(x)$ of $K$ is  logarithmically homogeneous of degree $\theta$ if 
	\[f(tx)=f(x)-\theta \log(t). \]
\end{definition}

The following property is an easy consequence of the definition of self-concordance and can be found in section 9.6 in \cite{Boyd}. 

\begin{prop}\label{prop: selfconcordanceprop}
	If $f_1,f_2$ are  self-concordant functions on $K\subseteq \mathbb{R}^n$. then the following functions are also self-concordant.
	\begin{enumerate}[\upshape (i)]
		\item $af$, for all $a\geq 1$
		\item $f_1+f_2$
		\item $f(Ax+b)$ for all $A\in \mathbb{R}^{n\times m}, b \in \mathbb{R}^n$.
	\end{enumerate}
\end{prop} 

The following theorem is adapted from Theorem~2.4.2, Theorem~2.4.4, and Proposition~2.4.1 in \cite{NN}. One can also found this in section 11.6 in \cite{Boyd}.
\begin{theorem}\label{sf}
	Let $K$ be a proper cone, i.e., $K$ is solid, convex, pointed and closed, in $\mathbb{R}^n$ and let $f$ be a $\theta$-logarithmically homogeneous self-concordant barrier for $K$. Then the Fenchel conjugate $f^*$ of $f$ is a $\theta$-logarithmically homogeneous self-concordant barrier for $-K^*$,i.e, the polar dual of $K$. Moreover, we have the interior of $-K^*$ to be  \[\interior(-K^*) =\{\nabla f(x): x\in \interior(K)\} ,\]
	and 
	\[
	f^*(x)+f(y)+ \theta\log (-x^\tp y)  \geq \theta \log (\theta )-\theta\]
	where the equality  holds for if and only if $x= t\nabla f(y)$ for some $t>0$.   
\end{theorem}

We prove the following theorem when a self-concordance function of the dual cones $(\cone{k})^*$ is available.
\begin{theorem}
	Let $g$ be a $\theta$-logarithmically  self-concordant barrier of the dual cone $(\cone{k})^*$. Also let $f(Y)=\sum_{(i_1,\dots, i_k)\in {I}(d,k)}g(\Trn{i_1 \cdots i_k}^d(Y)), Y\in \interior((\cone{d}_k)^*)$. Assuming $\nabla f\coloneqq x \mapsto \nabla f(x)$ is invertible, and $(\cone{d}_k)^*$ is a proper cone, the function $F(X)=-\tr(X(\nabla f)^{-1}(-X))-f((\nabla f)^{-1}(-X))$ is a $\theta \card(I(d,k))$-logarithmically self-concordant barrier for $\cone{d}_k$.
\end{theorem}
\begin{proof}
	We first show $f$ is a $\theta \card({I(d,k)})$-logarithmically self-concordant barrier of the dual cone $(\cone{d}_k)^*$ where $\card({I(d,k)})$ is the cardinality of $I(d,k)$. 
	
	The barrier property follows from the fact that the boundary of $(\cone{d}_k)^*$ are those $Y$s such that some of $\Trn{i_1 \cdots i_k}^d(Y)$ are on the boundary of $(\cone{k})^*$. 
	
	 To verify that $f$ is self-concordant, we only need to show that for all $X$ in the interior of $(\cone{d}_k)^*$, $V\in \mathbb{S}^n$, $\phi(t)= f(X+tV) = \sum_{(i_1,\dots, i_k)\in {I}(d,k)}g(\Trn{i_1 \cdots i_k}^d(X+tV))$ is self-concordant.
	 
	  By Proposition \ref{prop: selfconcordanceprop}, it is enough to show $g(\Trn{i_1 \cdots i_k}^d(X+tV))$ is self-concordant. 
	Since $X$ is in the interior, we know $\Trn{i_1 \cdots i_k}^d(X+tV)$ is indeed in $\interior((\cone{k})^*)$ for all small $t$ and so $g(\Trn{i_1 \cdots i_k}^d(X+tV))$ is self-concordant as $g$ is. This proves $f$ is self-concordant on $(\cone{d}_k)^*$.
	
	 From the following computation,
	\begin{align*}
	f(tY) & =\sum_{(i_1,\dots, i_k)\in {I}(d,k)}g(\Trn{i_1 \cdots i_k}^d(tY))\\
	& \marka{=} \sum_{(i_1,\dots, i_k)\in {I}(d,k)}g(\Trn{i_1 \cdots i_k}^d(Y)) - \theta \log t \\
	& = \sum_{i(i_1,\dots, i_k)\in {I}(d,k)}g(\Trn{i_1 \cdots i_k}^d(Y))- \card(I(d,k))\theta  \log t,
	\end{align*}
	where (a) is because $g$ is $\theta$-logarithmically homogeneous. We see $f$ is indeed  logarithmically homogeneous of degree $\card(I(d,k))\theta$.
	
	By Theorem~\ref{sf}, we know that $f^*(X)$ is indeed a $\theta $ -logarithmically homogeneous self-concordant barrier for $-\cone{n}_k$. The Fenchel-Young's inequality asserts that \[f^*(X)+f(Y)  \geq \tr(XY)\] and this becomes equality if $X=\nabla f(Y)$.
	
	 Since $\nabla f$ is invertible from $\interior((\cone{d}_k)^*)$ to its image which from Theorem~\ref{sf} is just $\interior(-\cone{d}_k)= -\interior(\cone{d}_k)$,  $\nabla f$ is bijective from the interior of the dual cone to the interior of  $-\cone{d}_k$. Thus the notation $(\nabla f)^{-1}$ always makes sense. We have \[f^*(X) = \tr(X(\nabla f)^{-1}(X))-f((\nabla f)^{-1}(X))\] and so $F(X)$ is indeed a $\theta \card(I(d,k))$-logarithmically self-concordant barrier of $\cone{d}_k$.
	 \end{proof}

The condition $\nabla f\coloneqq x \mapsto \nabla f(x)$ is invertible is satisfied when $f$ has positive definite Hessian (see Lemma \ref{lemma: hessianinvertibelgrad} in Appendix). This is the case for $(\psd{d})_k$. 
\begin{lemma}
	 The function $f(x)=\sum_{(i_1,\dots, i_k)\in {\comb{[d]}{k}}}- \log (\det(\Trn{i_1 \cdots i_k}^d(Y)))$  is a $k{n\choose k}$ -logarithmically homogeneous self-concordant, strictly convex barrier on $((\psd{d})_k)^*$ and has positive definite Hessian on the interior of $((\psd{d})_k)^*$.
	 
\end{lemma}
\begin{proof}
	The cone  $((\psd{d})_k)^*$ can be easily verified to be proper. We only need to show the Hessian is positive definite as other parts are due to $-\log (\det (S))$ is self-concordant for $S\in \interior(\psd{k})$.

	Since first order approximation of $f$ is 
	\begin{align*}
	f(Y+\alpha H) & = \sum_{i_1 \cdots i_k} -\log (\det(\Trn{i_1 \cdots i_k}^d(Y+\alpha H)))\\
	& = \sum_{(i_1,\dots, i_k)\in I(d,k)} -\log (\det(\Trn{i_1 \cdots i_k}^d(Y))) - \alpha \tr (\Trn{i_1 \cdots i_k}^d(Y)^{-1} \Trn{i_1 \cdots i_k}^d(H))+O(\alpha^2)
	\end{align*}
	
	The first order derivative is 
	\[f'(Y)= - \sum_{(i_1,\dots, i_k)\in I(d,k)} \Aug{i_1 \cdots i_k}^d (\Trn{i_1 \cdots i_k}^d(Y)^{-1}).\]
	
	Now if we approximate the derivative up to the first order, we have 
	\begin{align*}
	f'(Y+\alpha H) =  \sum_{(i_1,\dots, i_k)\in I(d,k)}- \Aug{i_1 \cdots i_k}^d (\Trn{i_1 \cdots i_k}^d(Y)^{-1}) + \alpha \Aug{i_1 \cdots i_k}^d (\Trn{i_1 \cdots i_k}^d(Y)^{-1}\Trn{i_1 \cdots i_k}^d(H)\Trn{i_1 \cdots i_k}^d(Y)^{-1}) +O(\alpha^2)\\
	\end{align*}
	
	Thus we see 
	\begin{align*}
	D^2f(X)[H,H] & = \tr(\sum_{(i_1,\dots, i_k)\in I(d,k)}\Aug{i_1 \cdots i_k}^d (\Trn{i_1 \cdots i_k}^d(Y)^{-1}\Trn{i_1 \cdots i_k}^d(H)\Trn{i_1 \cdots i_k}^d(Y)^{-1})H)\\
	& = \sum_{(i_1,\dots, i_k)\in I(d,k) }\tr( (\Trn{i_1 \cdots i_k}^d(Y)^{-1}\Trn{i_1 \cdots i_k}^d(H)\Trn{i_1 \cdots i_k}^d(Y)^{-1})\Trn{i_1 \cdots i_k}^d(H))\\
	& = \sum_{(i_1,\dots, i_k)\in I(d,k)}\tr( (\Trn{i_1 \cdots i_k}^d(Y)^{-\frac{1}{2}}\Trn{i_1 \cdots i_k}^d(H)\Trn{i_1 \cdots i_k}^d(Y)^{-\frac{1}{2}})^2)\\
	\end{align*} 
	where $D^2f[H,H]$ denotes the value of second differential of $f$ taken at $x$ along the direction $H,H$.
	The last term is greater than zero for non-zero $H$. This means that $f$ is strictly convex and its Hessian is positive definite.
\end{proof}

\section{Appendix} 

Here we prove a few results in the main text. We first prove a simple Lemma used in proving Theorem \ref{thm:equivalencesdfineq} which states the equivalence between standard form  and inequality form of $k$OCP, 

\begin{lemma}\label{lemma: geqzerokoc}. Suppose $\{\cone{k}\}_{k=1}^\infty$ satisfies the \ep thoroughly. 
	If $x\in \mathbb{F}^{d}$ and $d \geq k$ and $\cone{1}=\mathbb{R}_+$, then 
	\[ \diag(x) \in \cone{d}_k \iff x\geq 0,\] where $x\geq 0$ means each component of $x$ is greater or equal to $0$.
\end{lemma}
\begin{proof}
	If $\diag(x) \in \cone{d}_k$, then 
\[\diag(x) = \sum_{(i_1,\dots, i_k)\in \comb{[d]}{k}}\Aug{i_1 \cdots i_k}^d (M^{i_1 \cdots i_k}), \quad\text{and} \,M^{i_1 \cdots i_k} \in \cone{k}.\] The \ep and our assumption on $\cone{1}$ implies that the diagonal of $M^{i_1 \dots i_k}$ are nonnegative. Thus we have $x\geq 0$. 

Conversely, if $x\geq 0$, we can write 
\[\diag(x) = \sum_{i=1}^d \diag( x_i e_i)\] where $e_i$ is the $i$th standard vector in $\mathbb{F}^d$. Because of our assumption on $\cone{1}$ and the $\{\cone{k}\}_{k=1}^\infty$ satisfies the \ep thoroughly, each $\diag( x_i e_i) \in \cone{d}_k$ and this is a valid decomposition in $\cone{d}_k$. Thus $\diag(x)\in \cone{d}_k$.
\end{proof}

We note the assumption $\cone{1} = \mathbb{R}_+$ has no loss of generality since for nonempty one dimensional cone in $\mathbb{S}^1$, it is either $\mathbb{R}_-$ or $\mathbb{R}_+$.

The following Theorem includes Lemma \ref{lemma: sdd results} as a special case. See item (i) and (vi) of the theorem. The same result can also be found in \cite[Theorem 8,9]{boman2005factor} but we give a different proof. 

\begin{theorem}\label{thm5}
For a matrix $A =[a_{ij}]_{ij}$, denote $M(A)=[\alpha_{ij}]$ where $\alpha_{ii}=a_{ii} $ for all $i$ and $\alpha_{ij}=-|a_{ij}|$ for all $i \ne j$ and $\rho(A)=\max\{|\lambda|: \lambda\; \text{is an eigenvalue of}\; A \}$.  The following are all equivalent when $A \in \mathbb{S}^n$.
	\begin{enumerate}[\upshape (i)]
		\item $A\in \SDD{d}$ ;
		\item $M(A)\in \SDD{d}$; 
		\item there exists $D=\diag(d),d>0$, i.e., elementwise positive,  such that $D^\tp AD\in\DD{d}$;
		\item there exists a permutation matrix $P$ such that $P^\tp AP\in \SDD{d}$ ;
		\item $M(A)= sI-B$ for some $s$ and $B$ where $B$ is a non-negative matrix and $s$ is greater or equal to the largest absolute value of eigenvalue of $B$, i.e.,  $s\geq \rho(B)$;
		\item $M(A)$ is positive semi-definite.
	\end{enumerate} 
	
\end{theorem}

\begin{proof}
We begin with the equivalence between (i)--(iv). It directly follows from the definition that (i) and (ii) are equivalent. By multiplying out $D^\tp  A D$ and examining row by row, one finds the condition $D^\tp  A D\in \DD{d}$ is the same as $A\in \SDD{d}$. Thus (iii) is equivalent to (i).  The equivalence between (i) and (iv) can also be easily verified from the definition.
		
Next we show that (v) and (vi) are equivalent. First (v) implies (vi) since for symmetric matrix, $\|B\|_2= \rho(B)$ and $\max_{\|v\|_2=1}v^\tp Bv =\rho(B)$,  $\min_{\|v\|_2=1}v^\tp M(A)v = \min_{\|v\|_2=1}s -   \max_{\|v\|_2=1}v^\tp Bv = s-\rho(B)\geq 0$. Also, (vi) implies (v): If $M(A)$ is positive semi-definite, then we know all its eigenvalues are non-negative and the largest eigenvalue is positive (the case $M(A)$ is a zero matrix is trivially true for the implication). Denote the eigenvalue of $M(A)$ to be $\lambda_1\geq \lambda_2\geq \dots \geq \lambda_n\geq 0$ (counting multiplicity), then $\lambda_1 I- M(A) \in \psd{d}$. Furthermore, since $B=\lambda_1 I- M(A)$ is positive semi-definite, the diagonal element is non-negative and so $B$ is non-negative. We also have $\rho(B)=\lambda_1-\lambda_n\leq \lambda_1$. This shows (vi) implies (v).
		
Lastly we deduce that (v)--(vi)  and (i)--(iv) are equivalent. Suppose $A\in \SDD{d}$ and so is $M(A)$, then by characterization (iii) and the fact that diagonally dominant matrix are positive semi-definite which follows from Gerschigorin circle theorem, we see $M(A)$ is positive semi-definite. This shows (i)--(iv) implies (v)--(vi). Conversely, suppose $M(A)=sI- B$ where $B$ is non-negative and $s\geq \rho(B)$. Since $B$ is symmetric, there always exists a permutation matrix $P$ such that 
\[
		P^\tp BP= \begin{bmatrix}
		B_1 & & &\\
		& B_2 & & \\
		& & \ddots &\\
		& & & B_k
		\end{bmatrix},
\]
and $B_i$ are all irreducible and square matrices and for all $i$, $\rho(B_i)\leq s$. Now by the the Perron--Frobenius theorem, we know for each $B_i$, there is an elementwise positive vector $v_i$ such that $B_iv_i =\rho(B_i)v_i$. Then if we multiply the vector $v=(v_1,\dots, v_k)$ on the right to $P^\tp M(A)P$, we have $P^\tp M(A)Pv = sv - (\rho(B_1)v_1,\dots,\rho(B_k)v_k)\geq 0$. This shows that $P^\tp M(A)P\in \SDD{d}$ is and so are $M(A)$ and  $A$.
\end{proof}

  \begin{lemma}\label{lemma: hessianinvertibelgrad}
	Suppose $f$ is a real valued second order differentiable function defined on a open convex cone $K\subseteq \mathbb{R}^n$. If $f$ has positive definite Hessian, then $\nabla f$ is an injection. 
\end{lemma}

\begin{proof}
	For every $x\in K$ and $x+h\in K, h \ne 0$, we have 
	\[ h^\tp (\nabla f(x+h)-\nabla f(x))=\int_{0}^1 h^\tp  \nabla^2 f(x+th) h\, dt>0\] 
	as $\nabla ^2 f$, the Hessian, is positive definite. This means $f(x+h) \ne f(x)$ and $f$ is injective.
\end{proof}
\bibliographystyle{alpha}
\bibliography{reference}
\end{document}